\documentclass{amsart}
\usepackage[utf8]{inputenc}
\usepackage{graphicx}
\usepackage{amsmath,amssymb}
\usepackage{amsthm}
\usepackage{color}
\usepackage{tikz}
\usepackage[utf8]{inputenc}

\newtheorem{thm}{Theorem}[section]
\newtheorem{cor}[thm]{Corollary}
\newtheorem{lem}[thm]{Lemma}
\newtheorem{prop}[thm]{Proposition}

\theoremstyle{definition}
\newtheorem{defn}[thm]{Definition}
\theoremstyle{remark}
\newtheorem{rem}[thm]{Remark}

\numberwithin{equation}{section}

\newcommand{\del}{\delta}

\newcommand{\B}{{\mathbb B}}

\newcommand{\R}{{\mathbb R}}

\newcommand{\C}{{\mathbb C}}

\newcommand{\N}{{\mathbb N}}

\newcommand{\Z}{{\mathbb Z}}

\newcommand{\calA}{{\mathcal A}}
\newcommand{\calB}{{\mathcal B}}

\newcommand{\calD}{{\mathcal D}}

\newcommand{\calF}{{\mathcal F}}
\newcommand{\calG}{{\mathcal G}}

\newcommand{\frakC}{{\mathfrak C}}

\def\one{\mbox{1\hspace{-4.25pt}\fontsize{12}{14.4}\selectfont\textrm{1}}}

\makeatletter
\newcommand{\doublewidetilde}[1]{{%
  \mathpalette\double@widetilde{#1}%
}}
\newcommand{\double@widetilde}[2]{%
  \sbox\z@{$\m@th#1\widetilde{#2}$}%
  \ht\z@=.9\ht\z@
  \widetilde{\box\z@}%
}
\makeatother

\title[Dyadic decomposition]{Dyadic decomposition of convex domains of finite type and applications}

\date{\today}
\author{Chun Gan\quad Bingyang Hu \quad Ilyas Khan}

	\address{Chun Gan \\ University of Wisconsin--Madison \\480 Lincoln Drive\\ Madison, WI 53706, USA} \email{cgan5@wisc.edu}

	\address{Bingyang Hu \\ University of Wisconsin--Madison \\480 Lincoln Drive\\ Madison, WI 53706, USA} \email{bhu32@wisc.edu}
	
	\address{Ilyas Khan \\ University of Wisconsin--Madison \\480 Lincoln Drive\\ Madison, WI 53706, USA} \email{ikhan4@wisc.edu}



\begin{document}

\maketitle

\begin{abstract}
 In this paper, we introduce a dyadic structure on convex domains of finite type via the so-called dyadic flow tents. This dyadic structure allows us to establish weighted norm estimates for the Bergman projection $P$ on such domains with respect to Muckenhoupt weights. In particular, this result gives an alternative proof of the $L^p$ boundedness of $P$. Moreover, using extrapolation, we are also able to derive weighted vector-valued estimates and weighted modular inequalities for the Bergman projection. 
\end{abstract}

\section{Introduction}

Let $\Omega \subseteq \C^n$ be a smoothly bounded convex domain of finite type  and $H(\Omega)$ be the collection of all holomorphic functions on $\Omega$ with compact-open topology. Recall the \emph{Bergman space}
$$
A^2(\Omega):=\left\{ f \in H(\Omega) : \|f\|^2:=\int_\Omega |f(z)|^2dV(z)<\infty \right\},
$$
where $dV$ is the standard Euclidean volume form on $\C^n$. It is a well-known fact that the norm topology on $A^2(\Omega)$ is finer than the compact-open topology. Moreover, $A^2(\Omega)$ is a Hilbert space with the inner product
$$
\langle f, g \rangle:=\int_\Omega f(z) \overline{g(z)}dV(z), \quad f, g \in A^2(\Omega). 
$$
Let $K(z, \xi)$, $z, \xi \in \Omega$ be the \emph{Bergman kernel (reproducing kernel)} associated to the space $A^2(\Omega)$: Namely, 
$$
f(z)=\int_\Omega K(z, \xi)f(\xi)dV(\zeta), \quad  f \in A^2 (\Omega). 
$$
Recall that the \emph{Bergman projection} $P$, which is defined by 
\begin{equation} \label{Berproj}
P: f \mapsto \int_\Omega K(\cdot, \xi)f(\xi)dV(\xi),
\end{equation} 
is the Hilbert space orthogonal projection of $L^2(\Omega, dV)$ onto $A^2(\Omega)$. It is well-known that $P$ is bounded from $L^p(\Omega)$ to itself for $1<p<\infty$. The key idea used to prove this fact is to show that $P$ can be treated as a generalized Calder\'on-Zygmund operator (see, e.g., \cite{M94b}). 

Recently, using modern techniques of dyadic harmonic analysis, Rahm, Tchoundja and Wick \cite{RTW17} proved sharp weighted estimates for the Bergman projection and Berezin transform in Bergman spaces on the unit ball in $\C^n$, in terms of the Bekolle-Bonami constant of the weights.  We recall that a key observation in their argument is that the unit ball can be decomposed dyadically, which leads to a pointwise bound of the Bergman projection via a dyadic operator (or a positive sparse operator).

The purpose of this paper is to introduce a dyadic structure on $\Omega$, an arbitrary convex domain in $\C^n$ of finite type and smooth boundary (see, \eqref{modifieddyadic}).  Moreover, we show that this dyadic structure forms a Muckenhoupt basis on $\Omega$ (see, Proposition \ref{mucdyadic}). As a consequence, we generalize the pointwise sparse bounds in \cite{RTW17} to the case of convex domains of finite type (see, Lemma \ref{sparsebound}), and this allows us to estabish weighted norm estimates of the Bergman projection on $\Omega$ (see, Theorem \ref{Mainthm}) with respect to such dyadic structure. As corollaries, we are able to 
\begin{enumerate}
\item [(a).] provide a different proof of the $L^p$ boundedness of $P$ on $\Omega$ (see, Corollary \ref{cor001}). We remark that this result was first proved by McNeal in \cite{M94b}, where his approach was to show that the Bergman projection can be viewed as a generalized Calder\'on-Zygumund operator;
\item [(b).] derive new types of estimates of $P$ (see, Corollary \ref{cor002}), such as weighted vector-valued estimates and weighted modular inequalities. 
\end{enumerate}

The key ingredient in our approach is to study the following three different types of tents. More precisely, if $r$ is the the defining function of $b\Omega$, we use the mean curvature of the level sets of $r$ to show that these tents are equivalent to each other. This is new, and it gives an example of how several complex variables and dyadic calculus interact with each other.    
\begin{enumerate}
\item [(1).] \emph{McNeal-Stein tents}, which adapt well to the Bergman kernel of the domain $\Omega$ (see, Figure \ref{Figure1});
\item [(2).] \emph{Dyadic projection tents}, which connect the dyadic structure in $\Omega$ to the geometry of the boundary $b\Omega$ (see, Figure \ref{Figure2});
\item [(3).] \emph{Dyadic flow tents}, which are generated by the normal gradient flow of the defining function of $b\Omega$ and whose volumes can be calculated by tools from geometric analysis. As a consequence, these tents provide a dyadic structure and form a Muckenhoupt basis inside $\Omega$, so that dyadic calculus can be applied  (see, Figure \ref{Figure3}).
\end{enumerate}

\begin{figure}[ht]
\begin{tikzpicture}[scale=4]
\draw   plot[smooth,domain=-.7:.7] (\x, {\x*\x/3});
\draw (0, 0) node [below] {$Q$}; 
\draw [red, line width = 0.50mm] plot[smooth,domain=-.5:.5] (\x, {\x*\x/3});
\draw (-.7, .25) node[above] {$\Omega$};
\fill [red] (.5, .25/3) circle [radius=.2pt];
\fill [red] (-.5, .25/3) circle [radius=.2pt];
\draw (.5, .3)--(-.5, .3)--(-.5, -.3)--(.5, -.3)--(.5, .3); 
\fill [opacity=. 1, blue] (.5, .3)--(.5, .25/3)--(-0.5, .25/3)--(-.5, 0.3)--cycle;
\fill [opacity=. 1, blue] plot[smooth,domain=-.5:.5] (\x, {\x*\x/3})--(-.5, .25/3)--cycle;
\draw (.7, .4/3) node [below] {$b\Omega$};
\draw (.5, -.3) node [right] {$P_{\ell(Q)}(c(Q))$};
\end{tikzpicture}
\caption{McNeal-Stein tent}
\label{Figure1}
\end{figure}
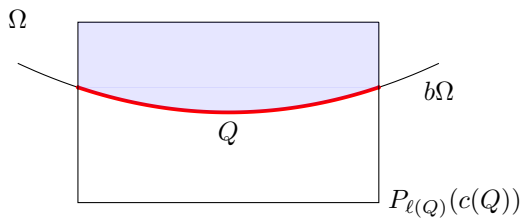

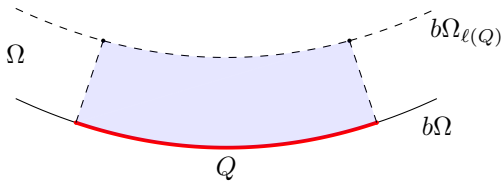
\begin{figure}[ht]
\begin{tikzpicture}[scale=4]
\draw   plot[smooth,domain=-.7:.7] (\x, {\x*\x/3});
\draw  [dashed] plot[smooth,domain=-.7:-.409175] (\x, {\x*\x/3+0.3});
\draw (.7, .4/3) node [below] {$b\Omega$};
\draw  [dashed] plot[smooth,domain=.409175:.7] (\x, {\x*\x/3+0.3});
\draw (0, 0) node [below] {$Q$}; 
\draw [red, line width = 0.50mm] plot[smooth,domain=-.5:.5] (\x, {\x*\x/3});
\draw (-.7, .25) node[above] {$\Omega$};
\fill (0.409175, 0.3558080602) circle [radius=.2pt];
\fill (-0.409175, 0.3558080602) circle [radius=.2pt];
\draw  [dashed] plot[smooth,domain=-.409175:.409175] (\x, {\x*\x/3+0.3});
\fill [red] (.5, .25/3) circle [radius=.2pt];
\fill [red] (-.5, .25/3) circle [radius=.2pt];
\draw [dashed](-.5, .25/3)--(-0.409175, 0.3558080602);
\draw [dashed] (.5, .25/3)--(0.409175, 0.3558080602);
\fill [opacity=. 1, blue] plot[smooth,domain=-.5:.5] (\x, {\x*\x/3})--(.5, .25/3)--(0.409175, 0.3558080602);
\fill [opacity=.1, blue] plot[smooth,domain=-.409175:.409175] (\x, {\x*\x/3+0.3})--(-.5, .25/3)--(-0.409175, 0.3558080602);
\draw (0.8, 0.46) node [below] {$b\Omega_{\ell(Q)}$};
\end{tikzpicture}
\caption{Dyadic projection tent}
\label{Figure2}
\end{figure}

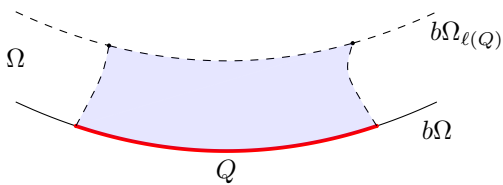
\begin{figure}[ht]
\begin{tikzpicture}[scale=4]
\draw   plot[smooth,domain=-.7:.7] (\x, {\x*\x/3});
\draw  [dashed] plot[smooth,domain=-.7:-.39] (\x, {\x*\x/3+0.3});
\draw (.7, .4/3) node [below] {$b\Omega$};
\draw  [dashed] plot[smooth,domain=.42:.7] (\x, {\x*\x/3+0.3});
\draw (0, 0) node [below] {$Q$}; 
\draw [red, line width = 0.50mm] plot[smooth,domain=-.5:.5] (\x, {\x*\x/3});
\draw (-.7, .25) node[above] {$\Omega$};
\fill (0.42, 0.1764/3+0.3) circle [radius=.2pt];
\fill (-0.39, 0.1521/3+0.3) circle [radius=.2pt];
\draw  [dashed] plot[smooth,domain=-.39:.42] (\x, {\x*\x/3+0.3});
\fill [red] (.5, .25/3) circle [radius=.2pt];
\fill [red] (-.5, .25/3) circle [radius=.2pt];
\draw [dashed] (.5, .25/3) .. controls (0.39, 0.1764/3+0.2)  .. (0.42, 0.1764/3+0.3);
\draw [dashed] (-.5, .25/3) .. controls (-0.41, 0.1521/3+0.2)  .. (-0.39, 0.1521/3+0.3);
\fill [opacity=. 1, blue] plot[smooth,domain=-.5:.5] (\x, {\x*\x/3})--(.5, .25/3)--(.5, .25/3) .. controls (0.39, 0.1764/3+0.2)  .. (0.42, 0.1764/3+0.3);
\fill [opacity=.1, blue]  plot[smooth,domain=-.39:.42] (\x, {\x*\x/3+0.3})--(-.5, .25/3)--(-.5, .25/3) .. controls (-0.41, 0.1521/3+0.2)  .. (-0.39, 0.1521/3+0.3); 
\draw (0.8, 0.46) node [below] {$b\Omega_{\ell(Q)}$};
\end{tikzpicture}
\caption{Dyadic flow tent}
\label{Figure3}
\end{figure}

The outline of this paper is as follows. In Section 2, we recall the construction of McNeal-Stein tents and the general theory of dyadic systems in a space of homogeneous type. In Section 3, we construct and study the properties of the dyadic projection tents and the dyadic flow tents quantitatively, and using these tents, we construct a dyadic structure on $\Omega$. We also show that such a dyadic structure forms a Muckenhoupt basis on $\Omega$. Section 4 is devoted to studying some weighted norm estimates of the Bergman projection $P$ and applications thereof.

Throughout this paper, for $a, b \in \R$, $a \lesssim b$ ($a \gtrsim b$, respectively) means there exists a positive number $C$, which is independent of $a$ and $b$, such that $a \leq Cb$ ($ a \geq Cb$, respectively). Moreover, if both $a \lesssim b$ and $a \gtrsim b$ hold, then we say $a \simeq b$.

\medskip
\section{McNeal-Stein tents and the dyadic structure decomposition on $b\Omega$}

We start by recalling some basic definitions. Let $\Omega$ be a smoothly bounded domain in $\C^n$. A point $p \in b\Omega$ is said to be of \emph{finite type} if the maximal order of contact at $p$ of $b\Omega$ with germs of non-singular complex analytic sets is finite. The domain $\Omega$ is said to be of \emph{finite type} if each $p \in \Omega$ is of finite type.  Moreover, we say $\Omega$ is of type $M$, for some $M \ge 0, M \in \Z$, if $M=\sup_{p \in b\Omega} \{ \textrm{type of} \ p\}$. Here are some remarks.

\begin{enumerate}
\item [(1).] In the above definition, if we also assume that $\Omega$ is convex, instead of considering all germs of all non-singular analytic sets, it is enough to take into account only complex lines (see, e.g., \cite{Y92});

\item [(2).] The finite type condition was discovered in the study of regularity of the $\bar{\partial}$-Neumann problem (see, e.g., \cite{C87, C89, K72, K73}), which is still in general open. 
\end{enumerate}

Throughout the paper, we always assume that $\Omega$ is a smoothly bounded open convex domain of type $M$ in $\C^n, n>1$, defined by a smooth function $r$, which is non-degenerate on $b\Omega$. We also assume that $r$ is everywhere convex. 

In this section, we first recall the McNeal-Stein tents, which were introduced by McNeal \cite{M94a}, and applied to study the behavior of the Bergman and Szeg\"o projections on such domains by McNeal and Stein later \cite{M94b, MS94}. These tents essentially capture all the geometric facts that motivate the construction of a dyadic system on $\Omega$. In particular, with these tents, we can make $b\Omega$ a space of homogeneous type. We refer the reader to \cite{M94a, M94b, MS94} for a detailed study of these tents, as well as of their applications. 

\begin{defn}
A \emph{space of homogeneous type} is an ordered triple $(X, \rho, \mu)$, where $X$ is a set, $\rho$ is a quasimetric, that is
\begin{enumerate}
    \item [(1).] $\rho(x, y)=0$ if and only if $x=y$;
    \item [(2).] $\rho(x, y)=\rho(y, x)$ for all $x, y \in X$;
    \item [(3).] $\rho(x, y) \le \kappa (\rho(x, z)+\rho(y, z))$, for all $x, y, z \in X$,
\end{enumerate}
for some constant $\kappa>0$, and the non-negative Borel measure $\mu$ is doubling, that is
$$
0<\mu(B(0, 2r)) \le D\mu(B(0, r))<\infty, \quad \textrm{for some} \ D>0, 
$$
where $B(x, r):=\{y \in X, \rho(x, y)<r\}$, for $x \in X$ and $r>0$.
\end{defn}

We start by recalling the construction of the \emph{$\varepsilon$-Extremal basis} $\{u_1, \dots, u_n\}$ at $\xi \in \overline{\Omega}$ with $\varepsilon>0$ small. 
\begin{enumerate}
\item[\textit{Step I}.] First, we let $q_1 \in \C^n$ such that $r(q_1)=r(\xi)+\varepsilon$ in the direction of the line given by the gradient $\partial r(\xi)$. Then $|q_1-\xi|$ is comparable to the distance from $\xi$ to the level set $bD_{\xi, \varepsilon}:=\{z \in \C^n: r(z)=r(\xi)+\varepsilon\}$. Let $u_1$ be the unit vector in the direction of $q_1-\xi$; 

\item[\textit{Step II}.] Second, we choose a unit vector $u_2$ in the complex orthogonal complement of the space $\langle u_1 \rangle$ such that the maximal distance from $\xi$ to $bD_{\xi, \varepsilon}$ along the directions orthogonal to $\langle u_1 \rangle$ is achieved along the line given by $u_2$ in a point $q_2$; 

\item[\textit{Step III}.] Repeat the second step by picking a unit vector $u_3$, which is orthogonal to the complex orthogonal complement of the space $\langle u_1, u_2 \rangle$ and maximizes the distance from $\xi$ to $bD_{\xi, \varepsilon}$ along these directions. The full basis will be constructed after $n$ steps.            
\end{enumerate}

\begin{defn} \label{MStent01}
Let $U$ be some fixed neighborhood of $b\Omega$. Given $\xi \in U$ and $\varepsilon>0$, and let $\{u_1, \dots, u_n\}$ be the $\varepsilon$-Extremal basis at $\xi$, the \emph{McNeal-Stein tent} associated to $\xi$ and $\varepsilon$ is defined as
$$
\Omega \cap P_{\varepsilon}(\xi),
$$
where 
\begin{equation}  \label{MStent}
P_\varepsilon(\xi):=\left\{\xi+\sum_{k=1}^n \lambda_{\xi, \varepsilon, k} u_k \in \C^n: |\lambda_{\xi, \varepsilon, k}| \le \tau(\xi, u_k, \varepsilon), k=1, \dots, n \right\},
\end{equation}
where for $u \in \C^n$, we set
$$
\tau(\xi, u, \varepsilon):=\sup\left\{c>0: |r(\xi+\lambda u)-r(\xi)| \le \varepsilon, |\lambda| \le c \right\}.
$$
\end{defn} 

\begin{rem}
There is a small issue that need to be clarified: it is pointed out in \cite{NPT13} (see, also, \cite[Remark 10.2]{Z16}) that \cite[Proposition 2.1, (iii)]{M94a} is problematic, in which, the concept of extremal basis was involved, however, it is known that this can be fixed by using ``minimal" bases introduced by Hefer \cite{H02} (see, e.g., \cite{NPT13}). 

We also point out that in this paper, we do not use  \cite[Proposition 2.1, (iii)]{M94a}, and the estimates obtained in terms of the extremal basis in Definition \ref{MStent}, as well as the geometric properties of McNeal-Stein tents, stay correct (see, e.g., \cite{H02, H04, NPT13}). 
\end{rem}

We summarize the properties of McNeal-Stein tent as below. 

\begin{enumerate}
\item [(1).] The symbol $\lambda_{\xi, \varepsilon, k}$ in \eqref{MStent} stands for the $k$-th coordinate of $z \in \C^n$ with respect to the $\varepsilon$-Extremal basis at the point $\xi \in U$.
\item [(2).] 
\begin{equation} \label{20191226eq01}
\tau(\xi, u_1, \varepsilon) \simeq \varepsilon \quad \textrm{and} \quad \tau(\xi, u_k, \varepsilon) \lesssim \varepsilon^{1/M}, k=2, \dots, n, 
\end{equation} 
where we recall that $M$ is the type of the domain $\Omega$, and the implict constants in the above estimates only depend on the defining function $r$ and the dimension $n$. Moreover, 
\begin{equation} \label{20191226eq02}
\tau(\xi, u_1, \varepsilon) \lesssim \tau(\xi, u_n, \varepsilon) \le \dots \le \tau(\xi, u_2, \varepsilon). 
\end{equation}
\item [(3).] The function $\rho: b\Omega \times b\Omega \to \R_+$ defines as 
\begin{equation} \label{metric}
\rho(\zeta_1, \zeta_2):=\inf \left\{ \varepsilon>0:  \zeta_1 \in P_{\varepsilon}(\zeta_2), \zeta_2 \in P_{\varepsilon}(\zeta_1) \right\} 
\end{equation} 
is a quasi-metric. Here, we still use $\kappa$ to denote the constant appearing in the triangle inequality, namely, 
$$
\rho(\zeta_1, \zeta_2) \le \kappa\left( \rho(\zeta_1, \zeta_3)+\rho(\zeta_2, \zeta_3) \right), \quad \forall \zeta_1, \zeta_2, \zeta_3 \in b\Omega. 
$$
\item [(4).] The volume form $d\sigma$ on $b\Omega$, which is the restriction of the Euclidean metric to $b\Omega$ is a doubling measure, namely, there exists a constant $K>0$, such that
$$
\sigma(B(\zeta, 2\varepsilon)) \le K\sigma(B(\zeta, \varepsilon)), \quad \forall \zeta \in b\Omega, \varepsilon>0,
$$
where 
$$
B(\zeta, \varepsilon):=\left\{ \eta \in b\Omega: \rho(\eta, \zeta)<\varepsilon \right\}=P_\varepsilon(\zeta) \cap b\Omega. 
$$
is the ball induced by the metric $\rho$ defined in \eqref{metric}. 
\end{enumerate}

Therefore, the triple $(b\Omega, \rho, \sigma)$ is a space of homogeneous type (SHT), and this allows us to apply the \emph{Hyt\"onen-Kairema} decomposition to get a dyadic system on $b\Omega$. More precisely, we have the following theorem.

\begin{thm}[{\cite[Theorem 2.1]{ACM15}, \cite{HK12}}] \label{dyadicSHT}
There exists a family of sets (\emph{we refer it as a dyadic grid in SHT}) $\calD=\bigcup\limits_{k \ge 1} \calD_k$, called a dyadic decomposition of $b\Omega$, constants $\frakC>0, 0<\del, \epsilon<1$, and a corresponding family of points $\{c(Q)\}_{Q \in \calD}$, such that
\begin{enumerate}
\item [(1).] $b\Omega=\bigcup\limits_{Q \in \calD_k} Q$, for all $k \in \Z$;
\item [(2).] For any $Q_1, Q_2 \in \calD$, if $Q_1 \cap Q_2 \neq \emptyset$, then $Q_1 \subseteq Q_2$ or $Q_2 \subseteq Q_1$;
\item [(3).] For every $Q \in \calD_k$, there exists at least one child cube $Q_c \in \calD_{k+1}$ such that $Q_c \subseteq Q$;
\item [(4).] For every $Q \in \calD_k$, there exists exactly one parent cube $\hat{Q} \in \calD_{k-1}$ such that $Q \subseteq \hat{Q}$;
\item [(5).] If $Q_2$ is a child of $Q_1$ , then $\sigma(Q_2) \ge \epsilon \sigma(Q_1)$;
\item [(6).] For any $Q \in \calD_k$, $B(c(Q), \del^k)  \subset Q \subset B(c(Q), \frakC\del^k)$.
\end{enumerate}
\end{thm}
We will refer to the last property as the \emph{sandwich property}. The sets $Q \in \calD$ are referred to as \emph{dyadic cubes} with center $c(Q)$ and sidelength $\ell(Q)=\del^k$, but we must emphasize that these are in general not cubes in $\R^n$ or any balls $B(\zeta, \varepsilon)$ induced by $\rho$. Moreover, we may also assume that $\del$ is sufficient small, more precisely, we require
\begin{enumerate}
    \item [(a).] $96\kappa^6\del \le 1$. Indeed, this can be achieved by replacing $\del$ by $\del^N$ and $\calD_k, k \in \Z$ by $\overline{\calD}_k:=\calD_{kN}, k \in \Z$,  for $N$ large such that $96\kappa^6\del^N \le 1$;

\item [(b).] $\del \le \frac{1}{100 C_\Omega}$, where $C_\Omega$ is some constant only depends on $\Omega$, such that
    $$
    \frac{1}{C_\Omega} |r(z)| \le \textrm{dist}(z, b\Omega) \le C_\Omega |r(z)|, \quad z \in \Omega. 
    $$
    \end{enumerate}

Such a choice of $\del$ allows us to treat all the balls $B(\zeta, \varepsilon)$ on $b\Omega$ in a dyadic way. 

\begin{prop}[{\cite[Theorem 4.1]{HK12}}] \label{TLT}
There exists a finite collection of dyadic grids $\calD^t, t=1, \dots, K_0$, such that for any ball $B=B(\xi, \varepsilon) \subset b\Omega$, there exists a dyadic cube $Q \in \calD^t$, such that
$$
B \subseteq Q \quad \textrm{and} \quad \ell(Q) \le \widetilde{\frakC}\varepsilon.
$$
Here,  $\widetilde{\frakC}$ is an absolute constant which only depends on $\kappa$ and $\del$. Moreover, the constants $\frakC_t, \del_t$ and $\varepsilon_t$ constructed in Theorem \ref{dyadicSHT} can be taken to be the same, that is, $\frakC_1=\frakC_2=\dots=\frakC_{K_0}$, $\del_1=\del_2=\dots=\del_{K_0}$ and $\varepsilon_1=\varepsilon_2=\dots=\varepsilon_{K_0}$. 
\end{prop}

\medskip
\section{Dyadic flow tents and the dyadic structure on $\Omega$}

We start with recalling the \emph{Carleson extension} from classical harmonic analysis: given a ball $B$ in $\R^n$ with radius $r>0$, the \emph{Carleson tent} over $\B$ is defined to be the ``cylindrical set"
$$
T(B):=\left\{(x, t) \in \R^{n+1}: x \in B, \ 0 \le t \le r \right\}.
$$
This simple construction plays an important role in studying various problems from harmonic analysis, for example, the $BMO$-functions, the $T(1)$ theorem and Kato's problems (see, e.g., \cite{G14b}). 

The goal of this section is to generalize the notion of Carleson extension to the convex domain of finite type and construct a collection of ``\emph{dyadic tents}" inside such domains quantitatively. Priorly, these dyadic tents give a partition of the domain $\Omega$, while the interesting feature for this partition is that it adapts to the Bergman kernel $K(\cdot, \cdot)$, more precisely, we have for any $z, \xi \in \Omega$, \begin{equation} \label{keyeq01}
|K(z, \xi)| \lesssim \frac{1}{\textrm{Vol}(T_d(Q))}, 
\end{equation}
where $Q$ is some dyadic cube on $b\Omega$ and $T_d(Q)$ is the dyadic tent associated to $Q$ that contains both $z$ and $\xi$,  and ``Vol" is the standard Euclidean volume form. Therefore, these dyadic tents provide the proper underlying geometry and dyadic structure of the Bergman projection. 

Here is some motivation on how to construct these dyadic tents. We start with some dyadic system $\calD$ on $b\Omega$ and we expect to mimic the construction in $\R^n$ for cubes from all the generations in $\calD$, so that the estimate \eqref{keyeq01} holds uniformly for all these cubes. However, since there are infinitely many ``scales" (that is, $\del^k, k \ge 1$, the sidelength of these cubes), the simple compactness argument for $\Omega$ is not strong enough to guarantee such a uniform estimate. For this purpose, one possibility is to borrow the idea from the scaling map technique in sub-Riemannian geometry, which was introduced by Nagel, Stein and Waigner in \cite{NSW85}, and later developed by Stein, Stovall and Street in \cite{SS18, S11, S12, SS12, SS13}. More precisely, let $Q \in \calD$ and 
$$
\Phi_{c(Q), \ell(Q)}: \B^{2n-1}(1) \rightarrow U_{c(Q)}
$$
be a coordinate chart (in particular, we may also assume $\Phi_{c(Q), \ell(Q)}$ is one-to-one), where $\B^{2n-1}(1)$ is the unit ball in $\R^{2n-1}$, $U_{c(Q)}$ is a neighborhood of $c(Q)$ and $Q \subset \Phi_{c(Q), \ell(Q)}(B^{2n-1}(1/2))$. Then we may define the extension of $Q$ by
$$
T(Q):=\Phi_{c(Q), \ell(Q)}\left(\left\{ (\eta, t) \in \R^{2n}: \eta \in \Phi_{c(Q), \ell(Q)}^{-1}(Q), \ 0 \le t \le \frac{1}{2n} \right\}\right).
$$
This motivates us to consider the tents induced by the normal gradient flow of $r$ around $b\Omega$, that is, the dyadic flow tents. Moreover, we are also able to show the dyadic flow tents are equivalent to the dyadic projection tents, which are constructed directly by extending the dyadic cube on $b\Omega$ inside $\Omega$, with respect to its sidelength.  

\subsection{Dyadic projection tents and dyadic flow tents}

We start by defining the dyadic tents with respect to a dyadic cube on $b\Omega$.

\begin{defn} \label{dyadictent01}
Let $Q \in \calD$ be a dyadic cube with $\ell(Q)$ sufficiently small. Then the \emph{dyadic projection tent} associated to $Q$ is defined as
$$
T^{proj}_d(Q):=\left\{z \in \Omega \cap U: \Pi^{proj}(z) \in Q, \ r(z) \in (-\ell(Q), 0) \right\},
$$
where $U$ is the neighborhood of $b\Omega$ fixed in Definition \ref{MStent01} and $\Pi^{proj}$ is the  nearest point projection that maps $\Omega \cap U$ to $b\Omega$. Note that we may assume $U$ is sufficiently small, so that $\Pi^{proj}$ is well-defined (see, e.g., \cite[Page 110]{H76}).  
\end{defn}  

Our next goal is to define the dyadic flow tents, which play a crucial role in this paper. We start by defining the projection induced by the negative gradient flow in short time. Throughout this paper, we may assume that $\frac{2}{3} \le |\nabla r| \le \frac{3}{2}$ on $U$ and in particular that $|\nabla r|=1$ on $b\Omega$.

Let $\varphi: V \times [0, \tau) \to U \cap \overline{\Omega}$ be the negative gradient flow associated to  the vector field $-\frac{\nabla r}{|\nabla r|^2}$. That is, $\varphi$ is the unique short time solution of the ODE
\begin{equation} \label{norgradflow}
\frac{d\varphi}{dt}(x, t)=-\frac{\nabla r}{|\nabla r|^2} \left(\varphi(x, t)\right), \quad x \in V, t \in [0, t_0), 
\end{equation}
with the initial condition $\varphi(x, 0) \equiv x$, where 
\begin{enumerate}
    \item[(a).]  $V \subset U$ is a neighborhood of $b\Omega$ and $U$ is the neighborhood of $b\Omega$ in Definition \ref{MStent01}; 
    \item[(b).]  $[0, t_0 )$ is a short time interval for some $t_0 >0$ sufficiently small. 
\end{enumerate}

\medskip

 The Picard-Lindel\"of theorem asserts that $\varphi_t := \varphi(\cdot, t) : V \rightarrow U \cup \Omega$ is a diffeomorphism onto its image $\varphi(V, t)$ for every $t \in [0, t_0)$. In particular, $\varphi$ is an isotopy from the boundary $b\Omega$ to the level sets 
 $$
 b\Omega_t := \{ z \in \C^n : r(z) = -t\},
 $$
 which are equal to the images $\varphi(b\Omega, t)$. Indeed, for any $x_0 \in b\Omega$ and $t \in [0, t_0)$, we have
\begin{eqnarray} \label{cal01}
r(\varphi(x_0, t))%
&=& r(\varphi(x_0, t))-r(\varphi(x_0, 0))= \int_0^\eta \frac{d}{dt} \left(r(\varphi(x, s))\right)ds \\
&=& \int_0^t \nabla r(\varphi(x_0, s)) \cdot \frac{d\varphi}{dt}(x, s)ds \nonumber \\
&=& -\int_0^t \nabla r(\varphi(x_0, s)) \cdot \frac{\nabla r(\varphi(x_0, s))}{|\nabla r(\varphi(x_0, s))|^2}ds \nonumber\\
&=& -t.  \nonumber
\end{eqnarray}

\begin{defn}
Let $\Pi^{flow}: V \to b\Omega$ be projection to the boundary $b\Omega$ along the flow lines of the vector field $-\frac{\nabla r}{|\nabla r|^2}$, given by
$$
\Pi^{flow}(x):=\varphi^{-1}_{r(z)}(z), \quad z \in V.
$$
In particular, it sends the level sets $b\Omega_t$ to $b\Omega$ for each $t \in [0, t_0)$ (see, \eqref{cal01}). 
\end{defn}

We are ready to define the dyadic flow tents using $\Pi^{flow}$. 

\begin{defn}
Given  a dyadic cube $Q \in \calD$ on $b\Omega$ with $\ell(Q)$ sufficiently small,  \emph{dyadic flow tent} associated to $Q$ is defined to be
$$
T_d^{flow}(Q):=\left\{z \in V: \Pi^{flow}(z) \in Q, r(z) \in (-\ell(Q), 0) \right\}. 
$$
Moreover, we also define $T_d^{flow}(b\Omega):=\Omega$.
\end{defn}

The goal now is to show that these dyadic flow tents give a nice dyadic structure inside $\Omega$. We start by understanding how to calculate the volume of such cubes. This plays an important role later as it allows us to view the collections of the dyadic flow tents as a sparse collection. 

To do this, we need some knowledge from geometric analysis. The setting is as follows. Let $Q \in \calD$ be a dyadic cube with center $c(Q) \in b\Omega$. We take $F: U_Q \subset \R^{2n-1} \to b\Omega$ be local chart near $c(Q)$, with
\begin{enumerate}
    \item [(1).] $U_Q$ and $F(U_Q)$ are diffeomorphic via $F$;
    \item [(2).] $Q \subseteq F(U_Q)$.
\end{enumerate}
Note that we may choose $\del$ as small as possible so that the above conditions hold. Next we define 
$$
\calF_t(p)=\calF(p, t):=\varphi(F(p), t), \quad p \in U_Q, t \in [0, t_0)
$$
to be the \emph{local parametrization} of the level sets $\calF(U_Q, t) \subset b\Omega_t$ smoothly indexed by $t$. Here are some basic properties for $\calF$.
\begin{enumerate}
    \item[(1).] $U_Q$ and $\calF_t(U_Q)$ are diffeomorphic via $\calF_t$; 
    \item [(2).] For any $p \in V$ and $t \in (0, t_0)$, 
    $$
    r(\calF(p, t))=-t.
    $$
    Indeed, this is an easy consequence of \eqref{cal01};
    \item [(3).] $\calF$ satisfies the evolution equation: for any $p \in V$ and $t \in (0, t_0)$, 
    \begin{equation} \label{20191218eq01}
    \frac{d \calF}{dt}(p, t)=-\frac{\nabla r}{|\nabla r|^2} \left(\calF(p, t) \right). 
    \end{equation}
    Indeed, by \eqref{norgradflow}, we have
    \begin{eqnarray*}
    \frac{d \calF}{dt}(p, t)%
    &=& \frac{d \varphi(F(\cdot), \cdot)}{dt}(p, t)=\frac{d\varphi}{dt}\left(F(p), t\right) \\
    &=& -\frac{\nabla r}{|\nabla r|^2}(\varphi\left(F(p), t\right)) \\
    &=& -\frac{\nabla r}{|\nabla r|^2} \left(\calF(p, t) \right). 
    \end{eqnarray*}
\end{enumerate}

Using the local parametrization $\calF$, we calculate the evolution of area elements along the normal gradient flow. Let us recall several basic definitions first, and for this part, we use standard geometric analysis notation for partial derivatives ( $\partial_i=\frac{\partial}{\partial x_i}$) and index notation for tensors (see, \cite{E04, J17}). For each $t \in [0, t_0)$, the \emph{metric} on $\calF(U_Q, t)$ is given by
$$
g_{ij, t}=\partial_i \calF_t \cdot \partial_j \calF_t
$$
for $1 \le i, j \le 2n-1$, the \emph{inverse metric} by 
$$
\left(g^{ij}_t \right)=\left(g_{ij, t}\right)^{-1}
$$
and the \emph{area element} of $\calF(U_Q, t)$ by 
$$
\sqrt{g_t}=\sqrt{\textrm{det} \left( g_{ij, t} \right)}.
$$
Let further, $\nu=\frac{\nabla r}{|\nabla r|}$ be the unit normal field to $b\Omega_t$ (note that the choice of $\nu$ is independent of the choice of the local parametrization $\calF$). The \emph{second fundamental form} of $\calF(U_Q, t)$ is defined by 
$$
A_{ij, t}:=\partial_i \nu \cdot \partial_j \calF_t=-\nu \cdot \partial_i \partial_j \calF_t, 
$$
for $1 \le i, j \le 2n-1$, and the \emph{mean curvature} is then defined by
$$
H_t:=\sum_{i=1}^{2n-1} \sum_{j=1}^{2n-1} g^{ij}_t A_{ij, t}. 
$$

\begin{rem}
\begin{enumerate}
    \item [(1).] It is clear that $A_{ij, t}=A_{ji, t}$ for any $t \in [0, t_0)$ and $1 \le i, j \le 2n-1$;
    \item [(2).] It is an important fact in geometric analysis that for each $t \in [0, t_0)$,  the mean curvature (more precisely, $H_t \circ \calF_t^{-1}$) is independent of the choice of local charts, that is, the choice of $\calF_t$ (see, e.g. \cite[Appendix A]{E04}). 
\end{enumerate}
\end{rem}

The proof of the following result is standard, and to be self-contained, we include the proof here for the convenience of the reader. Moreover, we remark that it is important for us that the implicit constant is independent of the choice of the cube $Q$ and the parametrization $F$.

\begin{prop} \label{evovolumeform}
There exists a constant $C>0$, independent of the choice of $Q$ and $F$, such that for any $t \in [0, t_0)$, the following estimate holds:
\begin{equation} \label{20191218eq04}
e^{-Ct} \sqrt{g_0} \le \sqrt{g_t} \le e^{Ct} \sqrt{g_0}. 
\end{equation}

\end{prop}

\begin{proof}
For each $1 \le i, j \le 2n-1$ and $t \in [0, t_0)$, we have
\begin{eqnarray*}
\partial_t \left(g_{ij, t} \right)%
&=& \partial_t \left(\partial_i \calF_t \cdot \partial_j \calF_t \right) \\
&=& \partial_t \partial_i \calF_t \cdot \partial \calF_j+ \partial_i \calF_t \cdot \partial_t\partial_j \calF_t \\
&=& \partial_i \partial_t \calF_t \cdot \partial \calF_j+ \partial_i \calF_t \cdot \partial_j\partial_t \calF_t \\
&=& \partial_i \left(-\frac{\nabla r}{|\nabla r|^2}\left(\calF_t \right) \right) \cdot \partial_j \calF_t+\partial_i \calF_t \cdot \partial_j  \left(-\frac{\nabla r}{|\nabla r|^2}\left(\calF_t \right) \right) \\
&& \quad \quad (\textrm{by \eqref{20191218eq01}.} ) \\
&=& \partial_i \left(-\frac{\nu}{|\nabla r|}\left(\calF_t \right) \right) \cdot \partial_j \calF_t+\partial_i \calF_t \cdot \partial_j  \left(-\frac{\nu}{|\nabla r|}\left(\calF_t \right) \right) \\
&=&-\frac{\partial_i \nu \cdot \partial_j \calF_t+\partial_i\calF_t \cdot \partial_j \nu}{|\nabla r|}\\
&& \quad \quad (\textrm{by the fact that $\nu \cdot \partial_j \calF_t=0, 1 \le j \le 2n-1$.}) \\
&=& -\frac{2A_{ij, t}}{|\nabla r|},
\end{eqnarray*}
which implies
\begin{eqnarray} \label{20191218eq02}
\partial_t \sqrt{g_t}%
&=& \frac{1}{2\sqrt{g_t}} \cdot \partial_t g_t \nonumber \\
&=& \frac{1}{2\sqrt{g_t}} \cdot \sum_{i=1}^{2n-1} \sum_{j=1}^{2n-1} g_t g^{ij}_t \partial_t(g_{ij, t}) \nonumber \\
&=& -\frac{\sqrt{g_t}}{|\nabla r|} \sum_{i=1}^{2n-1} \sum_{j=1}^{2n-1} g_t^{ij} A_{ij, t} \nonumber\\
&=&-\frac{H_t \sqrt{g_t}}{|\nabla r|}.
\end{eqnarray}
We claim that 
\begin{equation} \label{20191218eq03}
|H_t| \le C
\end{equation}
for $C>0$ independent of the choice of $Q, F$ and $t$. Indeed, by \cite[(3.6)]{G05}, we have
$$
H_t \circ \calF_t^{-1}=\frac{\left(\nabla r\right) \textrm{Hess}(r) \left(\nabla r\right)^T-|\nabla r|^2 \textrm{Trace}\left({\textrm{Hess}(r)}\right)}{2|\nabla r|^3},
$$
where $\textrm{Hess}(r)$ is the Hessian matrix of $r$. It is then clear that 
$$
|H_t \circ \calF_t^{-1} (z)| \le C,  \quad z \in \calF(U_Q, t), 
$$
for $C>0$ only depending on the $C^2$-norm of $r$ on $U$, and any dimension constants, which implies the desired claim. 

Therefore, combining \eqref{20191218eq02} and \eqref{20191218eq03}, we have the differential inequalities
$$
-2C \sqrt{g_t} \le \partial_t \sqrt{g_t} \le 2C \sqrt{g_t}. 
$$
The desired estimate \eqref{20191218eq04} then follows from an application of Gronwall's inequality. 
\end{proof}

As an application of the evolution of the volume form, we can estimate the volume of dyadic flow tents in a clean way. 

\begin{prop} \label{nearbouvolume}
There exists a $\tau_1>0$, independent of the choice of $Q$, such that for any $Q \in \calD$ and $t \in [0, \tau_1)$, the estimate 
$$
\frac{\sigma(Q)t}{2} \le \textrm{Vol} \left( \left\{z \in \Omega: \Pi^{flow}(z) \in Q, r(z) \in [-t, 0) \right\} \right) \le 2\sigma(Q)t. 
$$
holds. In particular, if $\ell(Q) \le \tau_1$, then
$$
\textrm{Vol}(T^{flow}_d(Q)) \simeq \sigma(Q)\ell(Q). 
$$
\end{prop}

\begin{proof}
This is a consequence of the coarea formula (see, e.g., \cite[Theorem 3.2.3]{F69} or \cite[Theorem 5.2.1]{KP08}). Denote 
$$
T_d(Q, t):=\left\{z \in \Omega: \Pi^{flow}(z) \in Q, r(z) \in [-t, 0) \right\},
$$
in particular, $T_d^{flow}(Q)=T_d(Q, \ell(Q))$. Then we have
\begin{eqnarray} \label{20191219eq01}
\textrm{Vol}(T_d(Q, t))%
&=& \int_{-t \le r(z) \le 0} \one_{\varphi(Q, [0, t))}(z)dV(z) \nonumber \\
&=& \int_{-t \le r(z) \le 0} \one_{\varphi(Q, [0, t))}(z) |\nabla r(z)|^{-1} |\nabla r(z)|dV(z) \nonumber \\
&=& \int_0^t\int_{b\Omega_s} \one_{\varphi(Q, s)}(z) |\nabla r(z)|^{-1}d\textrm{Vol}_{b\Omega_s}ds,
\end{eqnarray}
where $\one_S$ is the characteristic function of a given set $S$, and in the last equality, we apply the coarea formula and $dVol_{b\Omega_s}$ is the volume form on the level set $b\Omega_s$. Since $\frac{2}{3} \le |\nabla r| \le \frac{3}{2}$ on $U$, we see that for every $s \in (0, t]$, 
\begin{eqnarray} \label{20191219eq02}
\frac{2}{3} \int_{b\Omega_s} \one_{\varphi(Q, s)}(z)dVol_{b\Omega_s}%
&\le& \int_{b\Omega_s} \one_{\varphi(Q, s)}(z) |\nabla r(z)|^{-1}d\textrm{Vol}_{b\Omega_s} \nonumber \\
&\le& \frac{3}{2}  \int_{b\Omega_s} \one_{\varphi(Q, s)}(z)d\textrm{Vol}_{b\Omega_s}.
\end{eqnarray}
Applying the change of variable $z=\calF_s(p), p \in F^{-1}(Q) \subset \R^{2n-1}$, we see that
$$
\int_{b\Omega_s} \one_{\varphi(Q, s)}(z)d\textrm{Vol}_{b\Omega_s}=\int_{F^{-1}(Q)} \one_{F^{-1}(Q)}(p) \sqrt{g_s} dV_{2n-1}(p), 
$$
where $dV_{2n-1}$ is the standard Euclidean volume form on $\R^{2n-1}$. An application of Proposition \ref{evovolumeform} yields that
$$
e^{-Cs} \sigma(Q) \le \int_{b\Omega_s} \one_{\varphi(Q, s)}(z)d\textrm{Vol}_{b\Omega_s} \le e^{Cs} \sigma(Q), 
$$
where $C>0$ is the constant
Combining this estimate with \eqref{20191219eq01} and \eqref{20191219eq02}, we find that 
$$
\frac{2t e^{-Ct}}{3}\sigma(Q)\le \textrm{Vol}(T_d(Q, t)) \le \frac{3t e^{Ct}}{2}\sigma(Q).
$$
The desired result then follows if we choose $\tau_1>0$ sufficiently small, such that
$$
\frac{3e^{C\tau_1}}{2} \le 2 \quad \textrm{and} \quad \frac{2e^{-C\tau_1}}{3} \ge \frac{1}{2}.
$$
\end{proof}

We observe that as a consequence of Proposition \ref{TLT} and \cite[Lemma 2.7]{J10}, the dyadic projection tents and the McNeal-Stein tents are equivalent in the following sense: There exists a uniform constant $C_0>0$, such that for each $Q \in \calD$, 
$$
P_{C_0^{-1} \ell(Q)} (c(Q)) \cap \Omega \subseteq T_d^{proj}(Q) \subseteq P_{C_0 \ell(Q)}(c(Q)) \cap \Omega. 
$$
A natural question to ask is whether the dyadic flow tents are equivalent to the dyadic projection tents. The following result gives an affirmative answer to this question. Let us prove this result in a slightly more general setting. Let $\zeta \in b\Omega$ and $\varepsilon>0$. Define
$$
T_\varepsilon (\zeta) = \{ z \in \Omega:  \Pi (z) \in P_\varepsilon(\zeta) \cap b\Omega, \  r(z) \in (-\varepsilon, 0) \} 
$$
and 
$$
T^{flow}_\varepsilon (\zeta) = \{ z \in \Omega: \Pi^{flow} (z) \in P_\varepsilon(\zeta) \cap b\Omega, \  r(z) \in (-\varepsilon, 0)  \}. 
$$

\begin{thm} \label{projflow}
There exist constants $C_1>1$ and $\tau_2>0$ depending only on $r$ and $n$ such that for any $0<\varepsilon<\tau_2$ and $\zeta \in b\Omega$, 
$$
T_{\varepsilon}^{flow} (\zeta) \subseteq T_{C_1 \varepsilon} (\zeta) \quad \textrm{and} \quad  T_{ \varepsilon} (\zeta)\subseteq T^{flow}_{C_1\varepsilon}(\zeta). 
$$
In particular, for any $Q \in \calD$ with $\ell(Q) \le \frac{\tau_2}{\frakC}$, we have
$$
T_d^{proj}(Q) \subseteq T_{\frak C C_1 \ell(Q)}^{flow} (c(Q)) \quad \textrm{and} \quad T_d^{flow}(Q) \subseteq T_{\frak C C_1 \ell(Q)}^{proj} (c(Q)), 
$$
where $\frak C>1$ is the absolute constant defined in Theorem \ref{dyadicSHT}. 
\end{thm}

\begin{proof}
 We first prove that there exists $C_2>1$ such that 
$$
 T_\varepsilon^{flow} (\zeta) \subset T_{C_2\varepsilon}(\zeta)
$$
for any $\zeta \in b\Omega$ and $\varepsilon$ small.

Let $z \in T_\varepsilon^{flow}(\zeta)$, that is, 
$$
z=\varphi(\zeta', t),
$$
for some $\zeta' \in B(\zeta, \varepsilon)=P_{\varepsilon}(\zeta) \cap b\Omega$,  where $t=-r(z) \in [0, \varepsilon]$. Moreover, we let
$$
\{u_1, \dots, u_n\}
$$
be the $\varepsilon$-Extremal basis associated to $\xi$.

Observe that the normal projection of $z$ is given by the flow of $\zeta'$ along $b\Omega$ determined by the tangential component of the velocity of $\varphi(\zeta',t)$. More precisely, $\zeta'$ flows to $\Pi^{proj}(\varphi(\zeta',t))$ along the vector field
\[ \bigg( \frac{d}{dt}\varphi(\zeta',t) \bigg)^{tang} = \left(\frac{-\nabla r}{|\nabla r|^2} (\varphi(\zeta',t)) \right)^{tang},
\]
where the upper-script ``\emph{tang}" stands for the tangential component with respect to the tangent space $T_{\Pi^{proj}(\varphi(x,t))} b\Omega$. 

Since $\frac{2}{3} \le |\nabla r| \leq \frac{3}{2}$ on $U$, it follows that $$
\left|\frac{\nabla r}{|\nabla r|^2} \right|< 4, \ \textrm{on} \ U. 
$$
Therefore, $\Pi^{proj}(\varphi(\zeta',t))$ is displaced from $\zeta'$ by at most an Euclidean distance of $4t$, since the intrinsic distance in $b\Omega$ (a.k.a, the Riemannian distance in $b\Omega$ as a Riemannian submanifold of $\R^{2n}$) will always be longer than extrinsic distance in $\C^n$ (a.k.a, the Euclidean distance in $\C^n$). In particular, this suggests that the displacements from the point $\zeta'$ along the directions $\{u_1, \dots, u_n\}$ are at most $4t$ (see, Figure \ref{Figure4}). 

\begin{figure}[ht]
\begin{tikzpicture}[scale=5]
\draw   plot[smooth,domain=-.7:1] (\x, {\x*\x/3});
\draw (0, 0) node [below] {$\zeta$}; 
\draw [red, line width = 0.50mm] plot[smooth,domain=-.5:.5] (\x, {\x*\x/3});
\fill [red] (.5, .25/3) circle [radius=.2pt];
\fill [red] (-.5, .25/3) circle [radius=.2pt];
\fill (0, 0) circle [radius=.5pt];
\draw (.6, -.09) node [below] {$B(\zeta, \varepsilon)$}; 
\draw (.4, .16/3) node [below] {$\zeta'$}; 
\draw (-.4, .4) node [right] {$\Omega$}; 
\draw (-.7, .4/3) node [below] {$b\Omega$}; 
\fill (.4, .16/3) circle [radius=.5pt]; 
\fill (.5, 1/2.5) circle [radius=.5pt];
\draw [dashed] (.4, .16/3) .. controls (.37,.2)  .. (.5, 1/2.5);
\draw (.5, 1.05/2.5) node [above] {$\varphi(\zeta', t)$}; 
\fill (0.6, .36/3) circle [radius=.5pt];
\draw [dashed] (.5, 1/2.5)--(0.6, .36/3); 
\draw (.6, .2/3) node [right] {$\Pi^{proj}(\varphi(\zeta', t))$};
\end{tikzpicture}
\caption{}
\label{Figure4}
\end{figure}

Therefore, we see that
$$
\Pi^{proj}(\varphi(\zeta', t))= \zeta'+\sum_{k=1}^n \alpha_k u_k =\zeta+\sum_{k=1}^n (\lambda_{\zeta, \varepsilon, k}+\alpha_k)u_k,
$$
where for $\lambda_{\zeta, \varepsilon, k}, \alpha_k \in \C$ and
$$
|\lambda_{\zeta, \varepsilon, k}| \le \tau(\zeta, u_k, \varepsilon) \quad \textrm{and} \quad |\alpha_k| \le 4t \le 4\varepsilon, 
$$
for $k=1, \dots, n$. By \eqref{20191226eq01} and \eqref{20191226eq02}, there exists some constant $C'>0$ depending only on $n$ and $r$, such that
$$
(1/C') \varepsilon \le  \tau(z, u_1, \varepsilon) \le C' \varepsilon
$$
and
\begin{equation} \label{20191228eq01}
(1/C') \varepsilon \le  \tau(z, u_i, \varepsilon) \le C' \varepsilon^{1/M}, \;\; i = 2, \ldots, n,
\end{equation}
where we recall that $M$ is the type of the domain $\Omega$. Therefore, for each $k=1, \dots, n$, we have
 \begin{align*}
|\lambda_{\zeta, \varepsilon,k} + \alpha_k| &\le |\lambda_{\zeta, \varepsilon,k}| + |\alpha_k| \\
& \le \tau(z, u_k, \varepsilon) + 4\varepsilon \\
& \le \tau(z, u_k, \varepsilon) + 4C'\tau(z, u_k, \varepsilon)\\
&= (1 + 4C')\tau(z, u_k, \varepsilon), 
\end{align*}
which implies 
\begin{equation} \label{20191226eq03}
\Pi^{proj}(\varphi(\zeta',t)) \in (1+ 4C') P_\varepsilon (\zeta).
\end{equation}
Here, for any $\lambda>0$, the set $\lambda P_\varepsilon(\zeta)$ is the collection of points in $\C^n$ which have the representation
$$
\xi+\sum_{k=1}^n \lambda_{\xi, \varepsilon, k} u_k  \quad \textrm{with} \quad  |\lambda_{\xi, \varepsilon, k}| \le \lambda \tau(\xi, u_k, \varepsilon), \ k=1, \dots, n. 
$$
By \cite[Proposition 3.10]{DFF99} (see, also \cite[Proposition 2.1]{J10}), there exists a constant $C_2>0$, which only depends on $C'$, such that 
$$
(1+4C')P_{\varepsilon}(\zeta) \subset P_{C_2 \varepsilon}(\zeta),
$$
which, together with \eqref{20191226eq03}, implies $T^{flow}_\varepsilon(\zeta) \subseteq T_{C_2\varepsilon}(\zeta)$, uniformly in all $\zeta \in b\Omega$ and $\varepsilon \in (0, t_0)$, such that $C_2 \varepsilon<\tau_U$, where
$$
\tau_U:=\inf_{\zeta \in b\Omega} \max_{z \in U, \Pi^{proj} (z)=\zeta} |r(z)|. 
$$

\medskip

Next, we prove the other direction, that is, there exists $C_3>1$, such that
\begin{equation} \label{20191228eq10}
T_\varepsilon(\zeta) \subseteq T_{C_3\varepsilon}^{flow}(\zeta)
\end{equation}
for $\zeta \in b\Omega$ and $\varepsilon \in \left[0, \frac{\tau_1}{10C'} \right]$. 

Let $C_3>0$ be a large number to be chosen and we denote 
$$
\{u_{1, C_3}, \dots, u_{n, C_3}\}
$$
be the $\left(C_3\varepsilon\right)$-Extremal basis associated to $\zeta$.

Let $\zeta' \in b\Omega$. Much as before, the motion of $\Pi^{proj}(\varphi(\zeta',t)), \ t \in [0, \varepsilon]$ is given by the flow of the tangential vector field $\left(-\frac{-\nabla r}{|\nabla r|^2} (\varphi(\zeta',t))\right)^{tang}$.  In particular, if $\zeta' \not \in P_{C_3\varepsilon}(\zeta) \cap b\Omega$, then $\zeta'$ can be written 
\[\zeta' = \zeta + \sum_{k=1}^n \lambda_{z,C_3\varepsilon,k}u_{k, C_3}
\]
with some $k_0 \in \{1, \dots, n\}$ such that $|\lambda_{\zeta,C_3\varepsilon,k_0}| \ge \tau(\zeta, u_{k_0, C_3}, C_3 \varepsilon)$. As before, for any $t \in [0, \epsilon]$, the displacement $|\Pi^{proj}(\varphi(\zeta',t)) - \zeta'| \le 4\varepsilon$ and
\[\Pi^{proj}(\varphi(\zeta',t)) = \zeta + \sum_{k=1}^n (\lambda_{\zeta,C_3\varepsilon,k} + \beta_k) u_{k, C_3}
\]
where $|\beta_k|< 4\varepsilon$. Note that
\begin{align*}
|\lambda_{\zeta,C_3\varepsilon,{k_0}} + \beta_{k_0}| & \ge |\lambda_{\zeta,C_3\varepsilon,{k_0}}| - |\beta_{k_0}| \\
& \ge \tau(\zeta, u_{k_0, C_3}, C_3 \varepsilon) - 4\varepsilon \\
& \ge \tau(\zeta, u_{k_0, C_3}, C_3 \varepsilon) - \frac{4C' \tau(\zeta, u_{k_0, C_3}, C_3\varepsilon)}{C_3} \\
& \quad \quad \quad (\textrm{by \eqref{20191228eq01} with scale $C_3\varepsilon$.}) \\
& > \frac{\tau(z, u_{k_0}, C_3\varepsilon)}{2},
\end{align*}
where in the last inequality, we choose $C_3>10C'$. Therefore, with such a choice of $C_3$, 
$$
\Pi^{proj}(\varphi(\zeta',t)) \not \in \frac{1}{2} P_{C_3 \varepsilon}(\zeta) \cap b\Omega,
$$
which implies for each $t \in (0, \varepsilon)$, 
$$
\frac{1}{2}P_{C_3\varepsilon}(\zeta) \cap b\Omega \subseteq \Pi^{proj}(\varphi(P_{C_3\varepsilon}(\zeta) \cap b\Omega, t)),
$$
since the nearest projection $\Pi^{proj}$ is one-to-one between $b\Omega$ and $b\Omega_t$, for $t$ sufficiently small. Applying \cite[Proposition 3.1]{DFF99} (or \cite[Proposition 2.1]{J10}) again, we can take a constant $0<C_{1/2}<\frac{1}{10C'}$, which only depends on $r$ and $n$, such that
$$
P_{C_{1/2} C_3 \varepsilon}(\zeta) \subseteq \frac{1}{2}P_{C_3\varepsilon}(\zeta).
$$
Therefore, we conclude that
$$
P_{\varepsilon}(\zeta) \cap b\Omega \subseteq \Pi^{proj}(\varphi(P_{C_3\varepsilon}(\zeta) \cap b\Omega, t))
$$
for each $t \in [0, \varepsilon]$, if we take $C_3=\frac{1}{C_{1/2}}$. In other word, this implies that the nearest point projection of each layer between $0$ and $\varepsilon$  (with respect to the defining function $r$)  of $T_{C_3 \varepsilon}^{flow}(\zeta)$ contains $P_\varepsilon \cap b\Omega$, which is equivalent to \eqref{20191228eq10}. The constant $C_3$ is uniform for any $\zeta \in b\Omega$ and $\varepsilon \in \left[0, \frac{t_0}{10C'} \right]$.

The desired result then follows if we set 
$$
C_1=\max\{C_2, C_3\} \quad \textrm{and} \quad \tau_2=\min \left\{\frac{\tau_U}{C_2}, \frac{\tau_1}{10C'}\right\}.
$$
\end{proof}

 \subsection{Dyadic structure on $\Omega$.} 
 
The goal of the second part of this section is to construct the dyadic structure on $\Omega$.

Recall that from Theorem \ref{dyadicSHT}, we have a dyadic decomposition on $b\Omega$, that is
$$
\calD=\bigcup_{k \ge 1} \calD_k
$$
with parameters $\frakC>0$ and $0<\del, \epsilon<1$ with $\del$ sufficiently small. Take some $N_0 \in \N$ sufficient large, such that
\begin{enumerate}
    \item [(1).] $\Omega \backslash \left\{z \in \C^n: r(z) \le -\del^{N_0} \right\} \subseteq V$, where we recall that $V$ is the neighborhood of $b\Omega$ defined in \eqref{norgradflow}. Thus, for each $0 \le \eta \le \del^{N_0}$, there exists a diffeomorphism between the level set $b\Omega_\eta:=\{z \in \C^n: r(z)=-\eta\}$ and $b\Omega$, via the projection $\Pi^{flow}$;
    
    \item [(2).] $\del^{N_0} \le \min\left\{\tau_1, \tau_2\right\}$, where $\tau_1$ is defined in Proposition \ref{nearbouvolume} and $\tau_2$ is defined in Theorem \ref{projflow}. 
\end{enumerate}

We modify the dyadic system $\calD$ on $b\Omega$ (we still denote it by $\calD$) according to the choice of $N_0$ as follows:
\begin{equation} \label{modifieddyadic}
\calD:=\{b\Omega\} \cup \bigcup_{k \ge N_0} \calD_k. 
\end{equation}

\begin{rem}
With the modification \eqref{modifieddyadic}, one can easily check that the parameters $\frakC$ and $\del$ for the new dyadic decomposition remain the same as the original ones, while $\epsilon$ will become smaller, which only depends on the original $\epsilon$, $\del$, the defining function $r$, any dimension constants and $N_0$.
\end{rem}

For each $k \ge N_0$, we pair $\calD_k$ with the layer
$$
\Omega_k:=\{z \in \Omega, -\del^k \le r(z)<-\del^{k+1}\}. 
$$
Note that by the choice of $N_0$, we can find a partition of $\Omega_k$ by the cubes in $\calD_k$ via $\Pi^{flow}$, namely, we have
$$
\Omega_k=\bigcup_{Q \in \calD_k} W^{\textrm{up}}_Q,  
$$
where $W^{\textrm{up}}_Q:=\left\{z \in \Omega_k, \Pi^{flow}(z) \in Q\right\}$ for some $Q \in \calD_k$. It is clear that
$$
W_{Q_1}^{\textrm{up}} \cap W_{Q_2}^{\textrm{up}}=\emptyset, 
$$
for $Q_1 \neq Q_2$ and $Q_1, Q_2 \in \calD$. Therefore, we can decompose $\Omega$ into pairwise disjoint union by as follows
$$
\Omega=W_{b\Omega}^{\textrm{up}} \cup \bigcup_{Q \in \calD_k, k \ge N_0 } W_Q^{\textrm{up}},
$$
where $W_{b\Omega}^{\textrm{up}}:=\left\{ z \in \Omega, r(z)<-\del^{N_0} \right\}$. 
\begin{defn} \label{dyadictent}

\begin{enumerate}
    \item[(1).] For any $Q \in \calD_k$, we refer the set $W_Q^{\textrm{up}}$ the \emph{upper Whitney flow tent} associated to $Q$. Moreover, we denote $c(W_Q^{\textrm{up}})$ as the \emph{center} of $W^{\textrm{up}}_Q$, where $c(W_Q^{\textrm{up}})$ is the unique preimage of $c(Q)$ on $\{z \in \Omega, r(z)=-\del^k\}$ via $\Pi^{flow}$; 
    
    \item[(2).] There exists a partial order on the set $\left\{c\left(W_Q^{\textrm{up}}\right)\right\}_{Q \in \calD}$, namely, we say $c\left(W_{Q_1}^{\textrm{up}}\right) \le c\left(W_{Q_2}^{\textrm{up}} \right)$ if $Q_1$ is a descendant of $Q_2$. Moreover, we call
    $$
    \bigcup_{c\left(W_{Q'}^{\textrm{up}} \right) \le c\left(W_Q^{\textrm{up}} \right)} \left\{ c\left(W_{Q'}^{\textrm{up}} \right) \right\}
    $$
    the \emph{Bergman-flow tree} associated to the cube $Q$.
\end{enumerate}
\end{defn}

Note that for any $Q \in \calD$, $Q \neq b\Omega$, the  dyadic flow tent $T^{flow}_d(Q)$ can be decomposed as follows
    $$
    T^{flow}_d(Q):=\bigcup_{ c\left(W_{Q'}^{\textrm{up}} \right) \le c\left(W_Q^{\textrm{up}} \right)} W_{Q'}^{\textrm{up}}. 
    $$
    
Moreover, we see that another application of the coarea formula implies that
    \begin{equation} \label{upWhitest}
    \textrm{Vol} \left(W_Q^{\textrm{up}}\right) \le \textrm{Vol}\left(T^{flow}_d(Q)\right) \le  \frac{1}{4-\del} \textrm{Vol} \left(W_Q^{\textrm{up}}\right). 
    \end{equation}

The first estimate is clear. To see the second one, we let $Q \in \calD_k$ for some $k \ge N_0$. Observe that 
\[ \textrm{Vol}\left(T^{flow}_d(Q)\right) = \int_{-\delta^k}^0 \int_{b\Omega_s} \one_{\varphi(Q,s)} |\nabla r|^{-1} dVol_{b\Omega_s} ds\]
and
\[ \textrm{Vol}\left(W^{\textrm{up}}_Q\right) = \int_{-\delta^k}^{-\delta^{k+1}} \int_{b\Omega_s} \one_{\varphi(Q,s)} |\nabla r|^{-1} dVol_{b\Omega_s} ds \]
By Proposition \ref{nearbouvolume}, we know that 
\[ \frac{\delta^k(1-\delta) \sigma(Q)}{2} \le \textrm{Vol} \left(W_Q^{\textrm{up}}\right) \le  2\delta^k(1-\delta) \sigma(Q)\]
\[ \frac{\del^k\sigma(Q)}{2} \le \textrm{Vol}\left(T^{flow}_d(Q)\right) \le  2\delta^k \sigma(Q)\]
Thus,
\[ \textrm{Vol}\left(T^{flow}_d(Q)\right) \le  2\delta^k \sigma(Q) \le \frac{4}{1-\delta}  \textrm{Vol} \left(W_Q^{\textrm{up}}\right) .\]

The following lemma is crucial in proving the weighted estimates of Bergman projection. 

\begin{lem} \label{estdyadic}
For any $z, \xi \in \Omega$, there exists a dyadic cube $Q \in \calD^t$  for some $t \in \{1, \dots, K_0\}$ (see, Proposition \ref{TLT}), such that $z, \xi \in T^{flow}_d(Q)$ and 
$$
|K_\Omega(z, \xi)| \le \frac{A}{\textrm{Vol}\left(T^{flow}_d(Q)\right)},
$$ 
for some $A>0$ which only depends on the defining function $r$ and the dimension $n$.
\end{lem}

\begin{proof}
Let $z, \xi \in \Omega$. Without loss of generality, we may assume 
$$
-\del^{N_0} \le r(z), r(\xi)<0
$$
and
$$
\widetilde{\rho}(z, \xi):=\inf \left\{\del>0: z, \xi \in P_{\del}\left(\Pi^{proj}(z)\right) \right\} \le \del^{N_0+1}. 
$$
 Otherwise, we simply take $Q$ to be $b\Omega$ and in this case $T^{flow}_d(Q)=\Omega$. Here, we make a remark that $\widetilde{\rho}$ is indeed symmetric, that is, $\widetilde{\rho}(z, \xi) \simeq \widetilde{\rho}(\xi, z)$, where the implicit constant only depends on $r$ and $n$ (see, e.g., \cite[Page 194]{MS94}). 

By \cite[Lemma 2.7]{J10}, there exists some constant $0<C_4<\frac{1}{\del}$ (otherwise, we can choose our $\del$ smaller), such that
$$
P_{\widetilde{\rho}(z, \xi)}\left(\Pi^{proj}(z)\right) \cap \Omega \subset T_{C_4\widetilde{\rho}(z, \xi)}\left(\Pi^{proj}(z)\right)
$$
and
$$
\textrm{Vol} \left(P_{\widetilde{\rho}(z, \xi)}(\Pi^{proj}(z))\cap \Omega \right) \simeq \textrm{Vol} \left(T_{C_4\widetilde{\rho}(z, \xi)}(\Pi^{proj}(z))\right),
$$
where the implicit constant only depends on $r$ and $n$. Note that by \cite[Page 194]{MS94}, it holds that
\begin{enumerate}
\item [(1a).] $z, \xi \in P_{\widetilde{\rho}(z, \xi)}(\Pi^{proj}(z)) \cap \Omega$; 
\item [(2a).]
$$
|K_\Omega(z, \xi)| \lesssim \frac{1}{\textrm{Vol}\left(P_{\widetilde{\rho}(z, \xi)}(\Pi^{proj}(z))\cap \Omega \right)}.
$$ 
\end{enumerate}
Therefore, we have
\begin{enumerate}
\item [(1b).] $z, \xi \in T_{C_4\widetilde{\rho}(z, \xi)}(\Pi^{proj}(z))$; 
\item [(2b).]
$$
|K_\Omega(z, \xi)| \lesssim \frac{1}{\textrm{Vol}\left(T_{C_4\widetilde{\rho}(z, \xi)}(\Pi^{proj}(z))\right)}.
$$ 
\end{enumerate}
Note that
$$
C_4 \widetilde{\rho}(z, \xi)<\frac{1}{\del} \cdot \del^{N_0+1}=\del^{N_0} \le \tau_2, 
$$
and therefore, an application of Theorem \ref{projflow} implies that we have
\begin{enumerate}
    \item [(1c).] $z, \xi \in T_{\frac{C_4}{C_1}\widetilde{\rho}(z, \xi)}^{flow}(\Pi^{proj}(z))$;
    \item [(2c).] 
    $$
    |K_{\Omega}(z, \xi)| \lesssim \frac{1}{\textrm{Vol}\left(T_{\frac{C_4}{C_1}\widetilde{\rho}(z, \xi)}^{flow}(\Pi^{proj}(z))\right)}. 
    $$
\end{enumerate}

By Proposition \ref{TLT}, there exists some $Q \in \calD^t, t \in \{1, \dots, K_0\}$, such that 
$$
B\left(\Pi^{proj}(z), \frac{C_4}{C_1}\widetilde{\rho}(z, \xi)\right) \subset Q
$$
and 
$$
\ell(Q) \lesssim \frac{C_4}{C_1}\widetilde{\rho}(z, \xi).
$$
Moreover, we claim that 
\begin{equation} \label{20191119eq01}
\ell(Q) \ge \frac{C_4\widetilde{\rho}(z, \xi)}{5C_1\kappa \frakC}. 
\end{equation}
We prove it by contradiction. Take $\zeta_1, \zeta_2 \in B\left(\Pi^{proj}(z), \frac{C_4}{C_1}\widetilde{\rho}(z, \xi)\right)$, such that $\rho(\zeta_1, \zeta_2)=\frac{C_4\widetilde{\rho}(z, \xi)}{2C_1}$. However,
\begin{eqnarray*}
\rho(\zeta_1, \zeta_2)%
&\le& \kappa \left( \rho(\zeta_1, c(Q))+\rho(\zeta_2, c(Q)) \right) \\
&\le& 2\kappa\frakC \ell(Q) \le 2\kappa \frakC \cdot \frac{C_4 \widetilde{\rho}(z, \xi)}{5C_1 \kappa \frakC}\\
&<&\frac{C_4\widetilde{\rho}(z, \xi)}{2C_1},
\end{eqnarray*}
which is a contradiction. Moreover,without the loss of generality, we may assume $\frac{C_4}{C_1}$ is sufficiently large (otherwise, we may choose $\del$ smaller so that $C_4$ can be taken large enough), such that 
$$
|r(z)|+|r(\xi)|<\ell(Q). 
$$
Indeed, this follows from \eqref{20191119eq01} and the fact that $\widetilde{\rho}(z, \xi) \gtrsim |r(z)|+|r(\xi)|$ (see, \cite[Page 194]{MS94}).

Combining all these, we conclude that
\begin{enumerate}
    \item [(1d).] $z, \xi \in T^{flow}_d(Q)$; 
    \item [(2d).] 
    $$
    |K_\Omega(z, \xi)| \lesssim \frac{1}{\textrm{Vol}\left(T^{flow}_d(Q)\right)}. 
    $$
\end{enumerate}
The lemma is proved. 
\end{proof}

\begin{rem}
The dyadic structure $\calD$ defined in \eqref{modifieddyadic} (or the Bergman-flow tree) generalizes the Bergman tree used in \cite{ARS06, RTW17} for the unit ball $\B$ in $\C^n$.
\end{rem}

An important property of the collection $\calD$ is that it forms a Muckenhoupt basis of the domain $\Omega$. We start by recalling some basic setup from  \cite{CMP11}.

Recall that by a \emph{basis} $\calB$ of $\Omega$ we mean a collection of open sets contained in $\Omega$, and the maximal operator associated to $\calB$ is defined by 
$$
M_{\calB}f(z):=\sup_{z \in B \in \calB} \langle f\rangle_B
$$
if $z \in \Omega$ and $M_{\calB}f(z)=0$ otherwise, where for any $B \in \calB$, we denote
$$
\langle f \rangle_B:=\frac{1}{\textrm{Vol}(B)} \int_B |f(z)|dV(z). 
$$
Moreover, a \emph{weight} $w$ on $\Omega$ refers to a non-negative measurable function on $\Omega$. 

\begin{defn}
Let $w$ be a weight and $\calB$ be a basis on $\Omega$.
\begin{enumerate}
    \item [(1).] We say that $w$ belongs to \emph{the Muckenhoupt class associated to $\calB$, $A_{p, \calB}$, $1<p<\infty$}, if there exists a constant $K_p$ such that for every $B \in \calB$,
    $$
    \langle w\rangle_B  \left\langle w^{1-p'}\right \rangle_B ^{p-1} \le K_p<\infty, 
    $$
    where $p'$ is the conjugate of $p$. When $p=1$, we say that $w \in A_{1, \calB}$ if $M_{\calB}w(z) \le K_1 w(z)$ for almost every $z \in \Omega$. The infimum of all such $K_p$, denoted by $[w]_{A_{p, \calB}}$ is called the $A_{p, \calB}$ constant of $w$. Finally, we let
    $$
    A_{\infty, \calB}:=\bigcup_{p \ge 1} A_{p, \calB}.
    $$
    It is immediate form the definition that $w \in A_{p, \calB}$ if and only if $w^{1-p'} \in A_{p', \calB}$. Moreover, by H\"older's inequality, if $q>p$, then $A_{p, \calB} \subset A_{q, \calB}$.
    
    \medskip
    
    \item [(2).] A basis $\calB$ is a \emph{Muckenhoupt basis} if for each $p$, $1<p<\infty$, and for every $w \in A_{p, \calB}$, the maximal operator $M_{\calB}$ is bounded on $L^p(\Omega, w)$: for every $f \in L^p(w):=L^p(\Omega, w)$, 
    \begin{equation} \label{mucbasis}
    \int_\Omega M_{\calB}f(z)^p w(z)dV(z) \lesssim \|f\|_{p, w}^p:= \int_\Omega |f(z)|^pw(z)dV(z), 
    \end{equation}
    where the implicit constant is independent of $f$ and depends only on $[w]_{A_{p, \calB}}$, the defining function $r$ and the dimension $n$. 
\end{enumerate}
\end{defn}

\begin{prop} \label{mucdyadic}
$\calD$ is a Muckenhoupt basis for the domain $\Omega$. Consequently, 
\begin{equation} \label{Muckbasis2020}
\mathcal G:=\bigcup_{k=i}^{K_0} \calD^i
\end{equation}
is also a Muckenhoupt basis for $\Omega$, where we recall that the constant $K_0$ and the dyadic systems $\calD^1, \dots \calD^{K_0}$ are defined in Theorem \ref{TLT}. 
\end{prop}

\begin{proof}
The second claim is clear from the first one. The proof of the first claim is standard, and, for example, follows from an easy modification of the proof of \cite[Theorem 7.1.9]{G14a}. Here, we would like to make a remark that one possible modification of the proof would be to use the upper Whitney flow tent decomposition of $\Omega$, 
$$
\Omega=\bigcup_{Q \in \calD} W_Q^{up},
$$
to replace the Besicovitch type covering lemma \cite[Lemma 7.1.10]{G14a}.
\end{proof}

\medskip

\section{Sparse domination and sharp weighted estimates} 

In this section, we prove a pointwise sparse bound for the Bergman projection $P$ (see, \eqref{Berproj}) on $\Omega$. As a consequence, we establish weighted norm estimates of $P$, with respect to the Muckenhoupt weight. Furthermore, this gives us several new types of estimates of $P$, which include weighted vector-valued estimates and weighted modular inequalities. 

To start with, let us fix $\calG$ (see, \eqref{Muckbasis2020}) to be the Muckenhoupt basis associated to $\Omega$. Moreover, for any weight $w$ and $B \in \calB$, we denote 
$$
\langle f \rangle_B^w:=\left(\int_B w(z)dV(z)\right)^{-1} \int_B |f(z)|w(z)dV(z).
$$

We have the following lemma.

\begin{lem} \label{sparsebound}
Let $1<p<\infty$ and $f \in L^p(\Omega)$. Then
\begin{equation} \label{sparsebound01}
|(Pf)(z)| \lesssim \sum_{i=1}^{K_0} \sum_{Q \in \calD^i}  \langle f\rangle_{T_d^{flow}(Q)}\one_{T_d^{flow}(Q)}(z) , \quad z \in \Omega. 
\end{equation}
\end{lem}

\begin{proof}
Let $z \in \Omega$. By Lemma \ref{estdyadic}, we can find a finite collection of dyadic cubes $\calG_z \subset \calG$ (this is because there are only finitely many dyadic flow tents from $\calG$ containing $z$), such that 
$$
\Omega=\bigcup_{Q \in \calG_z} T_d^{flow}(Q).
$$
Moreover, for each $Q \in \calG_z$, we have $z \in T_d^{flow}(Q)$ and
$$
|K_{\Omega}(z, \xi)| \le \frac{A}{\textrm{Vol}\left(T_d^{flow}(Q) \right)}, \quad \xi \in Q, 
$$
where the constant $A>0$ is defined in Lemma \ref{estdyadic}. Note that for $Q_1, Q_2 \in \calG_z$ with $Q_1 \neq Q_2$, $T_d^{flow}(Q_1)$ may intersect with $T_d^{flow}(Q_2)$.  We have
\begin{eqnarray*}
|(Pf)(z)|%
&=& \left| \int_\Omega K_{\Omega}(z, \xi)f(\xi)dV(\xi) \right|\le \sum_{Q \in \calG_z} \int_{T_d^{flow}(Q)} \left|K_\Omega(z, \xi) f(\xi)\right|dV(\xi) \\
&\lesssim& \sum_{Q \in \calG_z} \frac{\one_{T_d^{flow}(Q)}(z)}{\textrm{Vol} \left(T_d^{flow}(Q) \right)} \int_{T_d^{flow}(Q)}|f(\xi)|dV(\xi),
\end{eqnarray*}
which is clearly bounded by the right hand side of \eqref{sparsebound01}.
\end{proof}

\begin{rem}
The expression on the right hand side of the estimate \eqref{sparsebound01} can be written as
$$
\sum_{i=1}^{K_0} \calA^{\calD^i}f (z), 
$$
where for each $i \in \{1, \dots, K_0\}$,
$$
\calA^{\calD^i}f (z):=\sum_{Q \in \calD^i}  \langle f\rangle_{T_d^{flow}(Q)}\one_{T_d^{flow}(Q)}(z), \quad z \in \Omega.
$$
The operators $\calA^{\calD^i}$ are known as the \emph{positive sparse operators} with respect to the collection of dyadic flow tents $\left\{T_d^{flow}(Q)\right\}_{Q \in \calD^i}$. Here, the word ``sparse" refers to the estimate \eqref{upWhitest}, more precisely, the estimate \eqref{upWhitest} guarantees the collection $\left\{T_d^{flow}(Q)\right\}_{Q \in \calD^i}$ to be a \emph{sparse family}.  We refer the reader \cite{L13a, L13b, CR16, LN19, L17} and the references there in for more information about this concept. 
\end{rem}

Here is the main result. 

\begin{thm} \label{Mainthm}
Let $1<p<\infty$ and $w \in A_{p, \calG}$. There holds
\begin{equation} \label{weightedestimate}
\|P: L^p(w) \to L^p(w)\| \lesssim [w]_{A_{p, \calG}}^{\max \left\{1, \frac{1}{p-1}\right\}}.
\end{equation}
Moreover, the estimate \eqref{weightedestimate} is sharp, in the sense that there exists a domain $\Omega$ of finite type, a $p \in (1, \infty)$ and a weight $w \in A_{p, \calG}$ such that
\begin{equation} \label{sharpexample}
\|P: L^p(w) \to L^p(w)\| \gtrsim [w]_{A_{p, \calG}}^{\max \left\{1, \frac{1}{p-1}\right\}}.
\end{equation}
Here, in both estimates above, the implicit constant only depends on $r$ and the dimension $n$.
\end{thm}

\begin{proof}
The sharpness \eqref{sharpexample} follows from the classical case when $\Omega$ is the unit ball in $\C^n$ and $p=2$, which was contained in \cite[Section 5]{RTW17}.

It suffices for us to prove \eqref{weightedestimate}. Let $\sigma=w^{-p'/p}$ be the dual weight of $w$. Observe that it suffices to consider the case when $1<p \le 2$, while for the case $p>2$ will follow from a duality argument (more precisely, the duality between $L^p(w)$ and $L^{p'}(\sigma)$) and the fact that $[\sigma]_{A_{p', \calG}}^{\frac{1}{p'-1}}=[w]_{A_{p, \calG}}$. 

Note that it suffices to prove
$$
\|P(\sigma \cdot): L^p(\sigma) \to L^p(w)\| \lesssim [w]_{A_{p,\calG}}^{\frac{1}{p-1}}, 
$$
that is
$$
\|P(\sigma f)\|_{p, w} \lesssim [w]_{A_{p, \calG}}^{\frac{1}{p-1}} \|f\|_{p, \sigma}
$$
for any $f \in L^p(\sigma)$, which is equivalent to proving
\begin{equation} \label{20200103eq01}
\|P(\sigma f)^{p-1}\|_{p', w} \lesssim [w]_{A_{p, \calG}} \|f\|_{p, \sigma}^{p-1}.
\end{equation}
Now let $g \in L^p(w)$ with $\|g\|_{p, w} \le 1$. Using Lemma \ref{sparsebound} together with the fact that $0<p-1 \le 1$, we find that
\begin{eqnarray*}
\left(P(\sigma f)(z)\right)^{p-1}%
&\le& \left(\sum_{i=1}^{K_0} \sum_{Q \in \calD^i}  \langle \sigma f\rangle_{T_d^{flow}(Q)}\one_{T_d^{flow}(Q)}(z)
\right)^{p-1} \\
&\le& \sum_{i=1}^{K_0} \sum_{Q \in \calD^i} \one_{T_d^{flow}(Q)}(z) \left( \langle \sigma f\rangle_{T_d^{flow}(Q)}\right)^{p-1}.
\end{eqnarray*}
Therefore, we have
\begin{eqnarray*}
&&\left|\int_\Omega  \left(P(\sigma f)(z)\right)^{p-1}g(z)w(z)dV(z) \right| \\
&& \le \sum_{i=1}^{K_0} \sum_{Q \in \calD^i} \left(\langle \sigma f\rangle_{T_d^{flow}(Q)}\right)^{p-1} \cdot \int_{T_d^{flow}(Q)} |g(z)w(z)|dV(z) \\
&&= \sum_{i=1}^{K_0} \sum_{Q \in \calD^i} \langle \sigma \rangle_{T_d^{flow}}^{p-1} \langle w \rangle_{T_d^{flow}} \textrm{Vol} \left( T_d^{flow}(Q) \right) \left( \langle f \rangle_{T_d^{flow}(Q)}^\sigma \right)^{p-1} \langle g \rangle_{T_d^{flow}(Q)}^w \\
&& \le [w]_{A_{p, \calG}} \sum_{i=1}^{K_0} \sum_{Q \in \calD^i} \textrm{Vol} \left( T_d^{flow}(Q) \right) \left( \langle f \rangle_{T_d^{flow}(Q)}^\sigma \right)^{p-1} \langle g \rangle_{T_d^{flow}(Q)}^w.
\end{eqnarray*}
We have the following claim: for each $i \in \{1, \dots, K_0\}$,
\begin{equation} \label{20200103eq02}
\sum_{Q \in \calD^i} \textrm{Vol} \left( T_d^{flow}(Q) \right) \left( \langle f \rangle_{T_d^{flow}(Q)}^\sigma \right)^{p-1} \langle g \rangle_{T_d^{flow}(Q)}^w \lesssim \|f\|_{p, \sigma}^{p-1}.
\end{equation} 
Indeed, by \eqref{upWhitest} and the disjointness of the upper Whitney flow tents with respect to each $\calD^i$, we have
\begin{eqnarray*}
\textrm{LHS of \eqref{20200103eq02}}%
&\lesssim& \sum_{Q \in \calD^i} \textrm{Vol} \left(W_Q^{up} \right) \left( \langle f \rangle_{T_d^{flow}(Q)}^\sigma \right)^{p-1} \langle g \rangle_{T_d^{flow}(Q)}^w \\
& \lesssim& \sum_{Q \in \calD^i} \int_{W_Q^{up}} \left(M_\calB^\sigma f(z) \right)^{p-1} M_B^w g(z) dV(z) \\
&\le& \int_\Omega \left(M_\calB^\sigma f(z) \right)^{p-1} M_{\calB}^w g(z) dV(z) \\
&\le& \|M_{\calB}^\sigma f\|_{p, \sigma}^{p-1} \|M_{\calB}^w g\|_{p, w} \\
&\lesssim & \|f\|_{p, \sigma}^{p-1},
\end{eqnarray*}
where in the above estimates, $M_{\calB}^w$ refers to the \emph{weighted dyadic maximal operator} with respect to the weight $w$, that is, 
$$
M_\calB^w g(z):=\sup_{B \in \calB} \left( \int_B w(z)dV(z) \right)^{-1} \int_B |g(z)|w(z)dV(z)
$$
and in the last estimate, we use the well-known fact that $M_\calB^w$ maps $L^p(w)$ strongly to $L^p(w)$ for $1<p<\infty$ (see, e.g., \cite[(7.1.28)]{G14a}). 

Therefore, by \eqref{20200103eq02}, we have
$$
\left|\int_\Omega  \left(P(\sigma f)(z)\right)^{p-1}g(z)w(z)dV(z) \right| \lesssim \|f\|_{p, \sigma}^{p-1}.
$$
The desired result will follow by taking the supremum of all $g \in L^p(w)$ with $\|g\|_{p, w} \le 1$.
\end{proof}

\begin{rem}
We note that the constant $[w]_{A_{p, \calG}}^{\max\left\{1, \frac{1}{p-1} \right\}}$ is exactly the constant appearing in the $A_2$ theorem (see, e.g., \cite{H12, L13a, LN19, L17}).
\end{rem}

Here are some consequences of Theorem \ref{Mainthm}. The following corollary was previously proved by McNeal in \cite{M94b} by showing that $P$ is a generalized Calder\'on-Zygmund operator.  

\begin{cor} \label{cor001}
For any $1<p<\infty$, $P$ is bounded on $L^p(\Omega)$. 
\end{cor}

\begin{proof}
The desired result is a special case of Theorem \ref{Mainthm} with $w \equiv 1$.
\end{proof}

Another application of Theorem \ref{Mainthm}, together with the fact that $\calG$ is a Muckenhoupt basis, is to establish some new estimates for the Bergman projection $P$.

\begin{cor} \label{cor002}
\begin{enumerate}
    \item [1.] (Weighted Vector-valued estimate). For $1<p, q<\infty$ and $w \in A_{p, \calG}$, we have
    $$
    \left\| \left( \sum_i |Pf_i|^q \right)^{1/q} \right\|_{p, w} \lesssim \left\| \left( \sum_i |f_i|^q\right)^{1/q} \right\|_{p, w}.
    $$
    In particular, it holds that
    $$
   \left\| \left( \sum_i |Pf_i|^q \right)^{1/q} \right\|_p \lesssim \left\| \left( \sum_i |f_i|^q\right)^{1/q} \right\|_p.
   $$
    \item [2.] (Weighted modular inequality). For $1<p<\infty$, $\alpha \in \R$ and $w \in A_{p, \calG}$, we have
    \begin{eqnarray*}
    &&\int_\Omega |Pf(z)|^p \log(e+|Pf(z)|)^\alpha w(z)dV(z) \\
    && \quad \quad \quad \quad  \quad \quad\lesssim \int_\Omega |f(z)|^p \log(e+|f(z)|)^\alpha w(z)dV(z).
    \end{eqnarray*}
\end{enumerate}
\end{cor}

\begin{proof}
These results follow from Proposition \ref{mucdyadic}, Theorem \ref{Mainthm} and the extrapolation theorem of Rubio de Francia (see, e.g., \cite{CMP11}). We leave the detail to the interested reader.
\end{proof}

\end{document}